\documentclass[12pt]{amsart}

\usepackage{amsmath,amssymb,amsthm,latexsym, xcolor}%
\usepackage{geometry}%
\usepackage{graphicx} 
\usepackage{color}  
\usepackage{ulem}
\begingroup%
\theoremstyle{plain}
\newtheorem{prop}{Proposition}[section]
\newtheorem{corollary}[prop]{Corollary}%
\newtheorem{lemma}[prop]{Lemma}%
\newtheorem{remark}[prop]{{Remark}}%
\newtheorem{definition}[prop]{{Definition}}%
\newtheorem{defn}[prop]{{Definition}}%
\newtheorem{example}[prop]{{Example}}
\endgroup%

\theoremstyle{plain}
\newtheorem{theorem}[prop]{Theorem}%

\theoremstyle{break}
\newtheorem{open?}[prop]{{Open Question}}

\theoremstyle{break}
\newcommand{\Min}{\operatorname{Min}}

\begin{document}

\title[Scale and tidy subgroups for Weyl-transitive automorphism groups]{Scale and tidy subgroups for Weyl-transitive automorphism groups of buildings}
\author{U. Baumgartner} 
\address{School of Mathematics and Statistics, 
University of Sydney, Sydney NSW 2006,
Australia
          Tel.: +61-2 9351 4221
           Fax: +61-2 9351 4534
}
\email{Udo.Baumgartner@sydney.edu.au}           
\author{J. Parkinson}
\address{School of Mathematics and Statistics, 
University of Sydney, Sydney NSW 2006,
Australia
          Tel.: +61-2 9351 4221
           Fax: +61-2 9351 4534
}
\email{James.Parkinson@sydney.edu.au}
\author{J. Ramagge}
\address{School of Mathematics and Statistics, 
University of Sydney, Sydney NSW 2006,
Australia
          Tel.: +61-2 9351 4533
           Fax: +61 2 9351 4534
}
\email{Jacqui.Ramagge@sydney.edu.au}

\subjclass[2010]{22D05, 52E24, 20E36, 20F65.} 
\keywords{totally disconnected locally compact groups, scale function, tidy subgroups, buildings}

\thanks{This research was supported by Australian Research Council grant  DP150100060.}

\date{\today}

\begin{abstract}
We consider
closed, Weyl-transitive groups 
of automorphisms 
of thick buildings. 
For each element of such a group,
we derive a combinatorial formula 
for its scale  
and 
establish the existence 
of a tidy subgroup 
for it 
that equals 
the stabilizer of 
a simplex. 
Simplices 
whose 
stabilizers are tidy 
for 
some element of the group
are 
characterized
in terms of 
the minimal set 
of the isometry 
induced by the element 
on the Davis-realisation 
of the building 
and 
in terms of the Weyl-distance 
between them 
and their image. 
We use our results 
to derive 
some topological properties 
of closed, Weyl-transitive groups 
of automorphisms. 
\end{abstract}

\maketitle

\section{Introduction}\label{sec:intro}

Tidy subgroups 
and the scale function 
provide 
important structural information 
on a totally disconnected locally compact group and have found diverse applications including the commensurated subgroup problem in arithmetic groups~\cite{SW:13}, and random walks and ergodic theory~\cite{DSW:06, JRW:96}.
Describing tidy subgroups 
for a given group 
and 
computing the scale of elements 
is 
however 
a challenging problem. 
For example, see \cite{Glo:98} for 
the computation of 
the scale function for $p$-adic Lie groups.
In a more general setting 
the main challenges in 
characterising tidy subgroups 
and 
calculating the scale function 
lie in either 
finding a description of the group that 
is suitable for the task, 
or conversely finding a way 
to express the concepts of tidyness and scale 
within the given description of the group. 

Every automorphism group 
of a locally finite cell complex 
is totally disconnected and locally compact 
in the compact-open topology. 
In this note 
we examine the example of 
the groups of automorphisms 
of a locally finite building 
with sufficiently transitive action. This important class of examples includes the subclasses of Lie groups over nonarchimedean local fields and Kac-Moody groups defined over finite fields. 
%
%
The methods used exploit the beautiful and rich geometry of the underlying building. Using this geometric approach we are able to provide 
a concrete description 
of the tidy subgroups, 
and 
a method 
to calculate the scale 
of an element 
directly in terms of 
the combinatorics 
of the action 
on the building. 
%

\section{Tidy subgroups and the scale}\label{sec:tidy&scale}

Every totally disconnected, locally compact group 
has 
compact, open subgroups by Van Dantzig's Theorem. 
Given a continuous automorphism, 
$\alpha$,  
of such a group
with continuous inverse, 
and a compact, open subgroup~$V$,
the index $|\alpha(V)\colon \alpha(V)\cap V|$ 
measures the distortion 
of~$V$ 
under~$\alpha$. 
This index is finite 
because
$\alpha(V)$ is again compact and open 
and  
all compact, open subgroups 
of a group 
are commensurable. 

Since the indices $|\alpha(V)\colon \alpha(V)\cap V|$ 
are integers, 
for every~$\alpha$  
there exist compact, open subgroups 
that minimize this index. 
These subgroups are called \emph{tidy for~$\alpha$}
and 
the minimal distortion index 
is called \emph{the scale of~$\alpha$}, 
denoted $s(\alpha)$.

In this paper 
we restrict 
our attention to 
inner automorphisms.  
A compact, open subgroup 
will be called 
\emph{tidy for} 
a group element~$g$ 
if and only if 
it is 
tidy for inner conjugation 
by~$g$, 
given by 
$x\mapsto gxg^{-1}$.  
Likewise 
the scale 
of inner conjugation
by~$g$ 
will be denoted~$s(g)$ 
and called 
\emph{the scale of~$g$}.  

There are several ways 
to characterize tidy subgroups.
A criterion 
for tidiness 
due to M\"oller 
will be 
the one 
we most often use in this paper. 

\begin{lemma}[{\cite[Corollary~3.5]{struc(tdlcG-graphs+permutations)}}]\label{lem:tidyness-crit_powers}
A compact, open subgroup~$U$ 
is tidy for an element~$g$ 
if and only if
\[
|U\colon U\cap g^{-n}Ug^n|=|U\colon U\cap g^{-1}Ug|^n \qquad \text{for all integers }n\geq 0\,.
\]
\end{lemma}

The following interpretation 
of tidy subgroups 
motivates 
our geometric approach 
to tidy subgroups.

\begin{lemma}[{\cite[Lemma~2(0)]{direction(aut(tdlcG))}}]\label{lem:metricCOS}
A compact, open subgroup 
is tidy 
for~$\alpha$ 
if and only if 
the displacement $\mathbf{d}(\alpha(O),O)$ 
with respect to 
the metric~$\mathbf{d}$ 
on the set of compact, open subgroups 
defined by $\mathbf{d}(V,W):=\log\left(|V\colon V\cap W|\cdot|W\colon W\cap V|\right)$
is minimal 
at~$O$. 
The value of 
this minimal displacement 
is $\log(s(\alpha))+\log(s(\alpha^{-1}))$.
\end{lemma}

The following 
formulas for the scale 
will be of central importance 
in section~\ref{sec:straight-displacements}.

\begin{lemma}
Suppose that~$O$ is tidy for~$g$. 
Then 
\[
s(g)=|gOg^{-1}\colon gOg^{-1}\cap O|=|O\colon O\cap g^{-1}Og|=|O\backslash OgO|
\]
\end{lemma}
\begin{proof}
The first formula 
for the scale 
is immediate from the definition. 
The second one 
follows by 
invariance of 
the subgroup index 
under conjugation. 
The third one 
can be verfied   
for example 
using Lemma~3.9 
in~\cite{graphTh-descr(scale-mult-semiGs(aut))} 
or Section~3.1 in~\cite{Bruhat-decomp&struc(HeckeR(padicChevalleyGs))}
\end{proof}

\section{Framework for buildings}\label{sec:buildings}

The theory of buildings grew from the fundamental work of Jacques Tits starting in the 1950s. The initial impetus was to give a uniform description of semisimple Lie groups and algebraic groups by associating a geometry to each such group. This ``geometry of parabolic subgroups'' later became known as the (spherical) building of the group~\cite{Tit:74}. The utility and scope of building theory has since far outgrown the original raison d'\^etre, with crucial applications in the theory of Lie groups defined over nonarchimedean local fields (the \textit{affine buildings}), and more broadly the theory of Kac-Moody groups. We will describe the latter connection in Example~\ref{ex:Kac} below.

Our main reference for the theory of buildings is~\cite{buildings_A&B}. Let us briefly fix notation. Let $(W,S)$ be a Coxeter system (with $|S| <\infty$) and let $\ell : W\to\mathbb{Z}_{\geq 0}$ be the length function on $W$ with respect to the generating set $S$. A building of type $(W,S)$ is a pair $(\Delta,\delta)$ where $\Delta$ is a nonempty set (whose elements are called chambers) and $\delta:\Delta\times\Delta \to W$ is a function (called the Weyl distance function) such that if $x,y\in\Delta$ then the following conditions hold, where $w=\delta(x,y)$:
\begin{enumerate}
\item[(1)] $w=1$ if and only if $x=y$.
\item[(2)] For each $s\in S$ there is $z\in\Delta$ with
$\delta(y,z)=s$ and $\delta(x,z)=ws$. 
\item[(3)] If $z\in\Delta$ with $\delta(y,z)=s\in S$ then $\delta(x,z)\in \{ws,w\}$. Moreover, if $\ell(ws)=\ell(w)+1$ then $\delta(x,z)=ws$. 
\end{enumerate}

The Coxeter group~$W$ 
is itself a building 
of type~$(W,S)$, 
with Weyl distance function given by $\delta(w,w')=w^{-1}w'$ (this building is called the \textit{Coxeter complex}).
Every subset of $\Delta$ which is $\delta$-isometric to the Coxeter complex is called an \emph{apartment} of $\Delta$, and a fundamental property of buildings is that given any two chambers $x,y\in\Delta$ there exists an apartment containing them both. 

A Coxeter system $(W,S)$ is called \textit{spherical} if $|W|<\infty$, and a building is called \textit{spherical} if its Coxeter system is spherical. In this paper we are primarily concerned with non-spherical buildings. However, as described below, certain spherical subbuildings play an important role. 

If $s\in S$, chambers $x,y\in\Delta$ are \textit{$s$-adjacent} (written $x\sim_s y$) if $\delta(x,y)=s$. The building $\Delta$ is:
\begin{enumerate}
\item[(1)] \textit{locally finite} if $|\{y\colon x\sim_s y\}|$ is finite for every $x\in\Delta$ and $s\in S$,
\item[(2)] \textit{thick} if $|\{y\colon x\sim_s y\}|\geq 2$ for every $x\in\Delta$ and $s\in S$,
\item[(3)] \textit{regular} if it is locally finite and $q_s=|\{y\colon x\sim_s y\}|$ is independent of~$x\in\Delta$.
\end{enumerate}
Henceforth, $(\Delta,\delta)$ denotes a non-spherical locally finite thick building of type $(W,S)$. Moreover, our assumptions on the group $G$ of automorphism of $\Delta$ (see below) will imply that $(\Delta,\delta)$ is also regular. In this case, if $w\in W$ then the value 
$$
q_w:=q_{s_1}\cdots q_{s_n}
$$
is independent of the particular reduced decomposition $w=s_1\cdots s_n$ for $w$ chosen (see, for example, \cite[Proposition~2.1]{Par:06}).

A \textit{gallery} of \textit{type} $(s_1,\ldots,s_n)$ from a chamber $x$ to a chamber $y$ is a sequence $x=x_0\sim_{s_1}x_1\sim_{s_2}\cdots \sim_{s_n}x_n=y$. We say that this gallery has \textit{length} $n$. A key property is that this gallery has minimal length amongst all galleries from $x$ to $y$ if and only if $\ell(s_1\cdots s_n)=n$. In other words, a gallery is minimal if and only if its type is reduced in~$W$. 

\begin{example}\label{ex:Kac}
Let $G=G(\mathbb{F})$ be Chevalley group over a field $\mathbb{F}$, with Coxeter system $(W,S)$. Let $B$ be a `Borel subgroup' (generated by the positive root subgroups and the torus elements). Then $(\Delta,\delta)$ is a spherical building of type $(W,S)$, where $\Delta=G/B$ and $\delta(gB,hB)=w$ if and only if $g^{-1}h\in BwB$. Thus $(\Delta,\delta)$ may be viewed as a combinatorial/geometric object encoding the structure of the Bruhat decomposition 
$
G=\bigsqcup_{w\in W}BwB, 
$ which highlights the original motivation for the invention of buildings. 

There is a far-reaching generalisation of this setup, where $G=G(\mathbb{F})$ is a Kac-Moody group over $\mathbb{F}$ with Coxeter system $(W,S)$. In this more general situation $\Delta$ is typically not spherical (that is, $|W|=\infty$). The building $(\Delta,\delta)$ is always thick, and it is locally finite and regular if and only if $\mathbb{F}=\mathbb{F}_q$ is a finite field, in which case $q_s=q$ for all $s\in S$. 
\end{example}

An \textit{automorphism} of $(\Delta,\delta)$ is a bijection $g:\Delta\to \Delta$ such that $\delta(g(x),g(y))=\delta(x,y)$ for all $x,y\in\Delta$. We will call such automorphisms \textit{type preserving} to distinguish from the slightly more general notion where diagram automorphisms of the underlying Coxeter system are permitted. 
A group $G$ of type preserving automorphisms is \textit{Weyl-transitive} if for all chambers $x,y,x',y'\in\Delta$ with $\delta(x,y)=\delta(x',y')$ there exists $g\in G$ with $g.x=x'$ and $g.y=y'$. Of course this implies that if $(\Delta,\delta)$ is locally finite then it is also regular. In Example~\ref{ex:Kac} each $g\in G$ acts as a type preserving automorphism on $\Delta=G/B$ in the obvious way, and the group $G/K$ acts Weyl transitively (where $K$ is the kernel of the action). 

Let $J\subseteq S$. Let $W_J=\langle \{s\colon s\in J\}\rangle$ be the \textit{standard parabolic subgroup of type $J$}. The \textit{$J$-residue} of a chamber $x\in\Delta$ is the set $R_J(x)=\{y\in\Delta\colon \delta(x,y)\in W_J\}$. It follows easily from the axioms that $(R_J(x),\delta|_{R_J(x)})$ is a building of type $(W_J,J)$. Let $\mathcal{R}$ be a residue and let $x\in\Delta$ be a chamber. There exists a unique chamber $y\in \mathcal{R}$ at minimal distance from $x$. This chamber is denoted $y=\mathrm{proj}_{\mathcal{R}}(x)$, and $\mathrm{proj}_{\mathcal{R}}:\Delta\to \mathcal{R}$ is called the \textit{projection onto $\mathcal{R}$}. See \cite[Section~5.3.2]{buildings_A&B} for basic properties of projections. 

Following \cite[Chapter~12]{buildings_A&B} there is a standard way to consider $(\Delta,\delta)$ as a simplicial complex by considering the partially ordered set of all spherical residues (that is, the $J$-residues with $|W_J|<\infty$). A simplex corresponding to a residue of type $J=\{s\}$ is called a \textit{panel} (of \textit{cotype} $s$). If $\Delta$ is locally finite then the associated simplicial complex is also locally finite (it is crucial here that only the spherical residues are considered), and thus its automorphism group totally disconnected, locally compact 
and also 
unimodular 
if it is Weyl-transitive 
(\cite[Corollary~5]{flatrk(AutGs(buildings))}) . 

The above simplicial complex has a natural geometric realisation $(X,d)$ as a $\mathsf{CAT}(0)$ space, called the \textit{Davis realisation} of the building  (see~\cite[Theorem 12.66]{buildings_A&B}). Let us briefly describe some geometric properties of this space. Firstly, each apartment $\mathbb{A}$ of $X$ is a $\mathsf{CAT}(0)$ realisation of the Coxeter complex of $(W,S)$. Thus there is a notion of \textit{walls} in $\mathbb{A}$ as the fixed point sets of the reflections in $W$ (that is, the elements $wsw^{-1}$ with $w\in W$ and $s\in S$). Each wall divides $\mathbb{A}$ into exactly two connected components (called \textit{halfspaces}), and a wall is said to \textit{separate} points $x,y\in \mathbb{A}$ if $x$ and $y$ do not lie in a common halfspace. See, for example, \cite{strongTitsAlt(<CoxeterG)} and~\cite{asymp-behave(word-metrics>CoxeterGs)} for further details. 

A \textit{wall} in the Davis realisation $X$ is a wall in some apartment of $X$, and similarly a \textit{halfspace} in $X$ is a halfspace in some apartment of~$X$. We say that a wall in $X$ \textit{separates} points $x,y\in X$ if this wall lies in some apartment containing $x$ and $y$ and separates these points in this apartment. Any geodesic in $X$ 
having a nondegenerate piece 
in a wall 
lies entirely in 
the wall, and
halfspaces and walls 
are convex. Moreover, 
two points 
in $X$ 
are separated by  
a wall 
if and only if 
their \textit{carriers} (defined as 
the simplex 
whose residue is 
the set of chambers 
whose realisation 
contains the point; see \cite[Definition~12.19]{buildings_A&B}), 
are separated by the wall.
In particular, we have:

\begin{lemma}\label{lem:separating_walls-compatibility}
The set of walls 
separating points 
in the Davis realisation 
equals 
the set of walls 
separating 
the supporting simplices 
of those points.  
%
\end{lemma}

%
%

Recall that $\Delta$ is assumed to be locally finite, thick, and regular. If $\pi$ is a panel of cotype $s$, we define the \textit{thickness of $\pi$} to be $q_{\pi}=q_s$. Thus the residue of $\pi$ contains $q_{\pi}+1$ chambers. If $m$ is a wall of the Davis realisation, we define the \textit{thickness of $m$} to be $q(m)=q_{\pi}$ for any panel $\pi$ contained in $m$. 
\label{page:well-def(wall-thickness)}
This is well defined, because if $\pi$ and $\pi'$ are panels contained in $m$ with associated residues $\mathcal{R}$ and $\mathcal{R}'$ then it is easily seen that $\mathrm{proj}_{\mathcal{R}}:\mathcal{R}'\to\mathcal{R}$ and $\mathrm{proj}_{\mathcal{R}'}:\mathcal{R}\to\mathcal{R}'$ are mutually inverse bijections. The \textit{thickness of a gallery $\gamma$} of type $s_1,\ldots,s_n$ is $q(\gamma)=q_{s_1}\cdots q_{s_n}$. Note that if $w=s_1\cdots s_n$ is a reduced expression then $q(\gamma)=q_w$. The \textit{thickness of a pair $x,y$} of simplices is defined to be
\begin{align*}
q(x,y)=\prod_{m\in \mathcal{M}_{\mathbb{A}}(x,y)} q(m),
\end{align*}
where $\mathbb{A}$ is an apartment containing $x$ and $y$, and $\mathcal{M}_{\mathbb{A}}(x,y)$ denotes the set of walls of $\mathbb{A}$ separating $x$ and $y$. This value is independent of the particular apartment~$\mathbb{A}$ containing~$x$ and~$y$ that we choose (this can be easily proven by considering the ``convex hull'' of the pair $x,y$, and noting that this convex hull is contained in every apartment that contains both~$x$ and~$y$). If~$x$ and~$y$ 
are singleton sets we write 
$q(x,y)$ instead of $q(\{x\},\{y\})$.  

%
Every isometry 
of~$X$ 
is semisimple 
by~\cite[Theorem~A]{semisimpl(polyh-Isom)}. 
In other words, 
for every isometry, 
$g$ say, 
of~$X$,  
the infimum 
in the following definition of 
the translation length 
of~$g$, 
written~$|g|$,  
is attained 
at some point 
of~$X$. 
\begin{equation}\label{eq:displacement-function}
|g|:=\inf\{d(x,g.x)\colon x\in X\}
\end{equation}
The isometry 
induced by $g\in G$ 
is elliptic 
if and only if  
the scale of~$g$ is~$1$ 
and 
is hyperbolic otherwise 
by \cite[Theorem~7 and Corollary~10]{flatrk(AutGs(buildings))}.

\begin{definition}
Let~$g$ 
be an 
an isometry 
of ~$X$. 
The \emph{minimal set\/}  
$\Min(g)$, 
of~$g$  
is 
the subset 
of~$X$ 
defined by 
\begin{equation*}
\Min(g):= \{x\in X\colon d(x,g.x)=|g|\}
\end{equation*} 
\end{definition}

The set
$\Min(g)$ is 
closed, convex, and
non-empty. 
%
The set $\Min(g)$ 
is 
the set of fixed points 
of~$g$ 
if~$g$ 
is elliptic 
and 
equals 
the union of 
the 
axes 
of~$g$ 
if~$g$ is hyperbolic. 

\section{Existence of a simplex with tidy stabilizer}\label{sec:Ex(tidy-simplex-Stab)}

For the computation of 
the scale of 
an element 
in~$G$ 
we need 
a generalization of 
Proposition~4 in~\cite{flatrk(AutGs(buildings))}. 
The following lemma 
is needed 
in its proof. 

\begin{lemma}\label{lem:Stab-transitive(pairs(cells)_givenWdist)}
Suppose that 
the action of a group $G$ 
on a building 
with $W$-distance $\delta$ 
is Weyl-transitive and type-preserving 
and 
let~$a$ and~$b$ be 
two simplexes  
whose residues 
are~$\mathcal{A}$ and~$\mathcal{B}$ 
respectively. 
Then
\begin{enumerate}
\item 
the subgroup~$G_b$ 
acts transitively on 
those simplices 
of the same type 
as~$a$, 
whose residues~$\mathcal{R}$  
satisfy 
$\min\bigl(\delta(\mathcal{B},\mathcal{R})\bigr)=\min\bigl(\delta(\mathcal{B},\mathcal{A})\bigr)$; 
\item 
the subgroup~$G_a\cap G_b$ 
acts transitively 
on~$\operatorname{proj}_\mathcal{A}(\mathcal{B})$, 
$\operatorname{proj}_\mathcal{B}(\mathcal{A})$ 
and 
the ordered pairs 
of chambers 
from these sets 
at Weyl-distance~$\min\bigl(\delta(\mathcal{A},\mathcal{B})\bigr)$. 
\end{enumerate}
\end{lemma}
\begin{proof}
We will use 
standard properties 
of projections 
between residues, 
compare for example 
section~5.3 in~\cite{buildings_A&B}. 
We 
first 
prove~(1). 
Let~$a'$ be 
a simplex  
fitting the description 
given 
in the statement 
of our claim 
and 
let~$\mathcal{A}'$ be 
the corresponding residue. 
Denote by~$c'$ 
some chamber 
in~$\operatorname{proj}_\mathcal{B}(\mathcal{A}')$ 
and 
let~$d':=\operatorname{proj}_{\mathcal{A}'}(c')$. 
Likewise 
choose 
a chamber, 
$c$ say,  
in~$\operatorname{proj}_\mathcal{B}(\mathcal{A})$ 
and 
let~$d:=\operatorname{proj}_{\mathcal{A}}(c)$. 
By our assumption 
on~$a'$ 
we have  $\delta(c,d)=\delta(c',d')=\min\bigl(\delta(\mathcal{B},\mathcal{A})\bigr)$. 
By Weyl-transitivity, 
$(c,d)$ 
can 
be mapped 
to~$(c',d')$ 
by some group element, 
which 
necessarily 
belongs to~$G_b$, 
proving claim~(1).

We 
next 
prove~(2). 
Since 
the last statement 
clearly 
implies 
the others, 
we 
restrict ourselves to 
proving it. 
To that end, 
choose two 
ordered pairs 
of chambers 
in~$\operatorname{proj}_\mathcal{A}(\mathcal{B})\times \operatorname{proj}_\mathcal{B}(\mathcal{A})$. 
The Weyl-distance 
between 
the first and the second element 
of the pair 
is~$\min\bigl(\delta(\mathcal{A},\mathcal{B})\bigr)$ 
for both pairs. 
By Weyl-transitivity, 
the pairs 
can 
thus 
be mapped 
to each other 
by some group element, 
which 
necessarily 
belongs to~$G_a\cap G_b$, 
proving the claim. 
The proof is complete. 
\end{proof}

Next 
we define 
a description of 
the relative position 
of two simplices 
with respect to 
the walls in the building 
which will play 
an important role 
in what follows. 

\begin{definition}
Let~$b$ and~$a$ be 
two 
(not necessarily distinct) 
simplices 
in a building 
with corresponding residues~$\mathcal{B}$ and~$\mathcal{A}$. 
We say 
\begin{enumerate}
\item  
the ordered pair~$(b,a)$ is 
\emph{aligned\/} 
if and only if 
$\operatorname{proj}_\mathcal{B}(\mathcal{A})=\mathcal{B}$. 
\item 
$a$ and~$b$ are 
\emph{aligned\/} 
if and only if 
$\operatorname{proj}_\mathcal{B}(\mathcal{A})=\mathcal{B}$ 
and 
$\operatorname{proj}_\mathcal{A}(\mathcal{B})=\mathcal{A}$. 
\end{enumerate}
\end{definition}

The pair~$(b,a)$ 
is aligned 
if and only if 
every wall 
that contains~$b$ 
and 
belongs to 
an apartment 
that contains 
both~$b$ and~$a$ 
does also contain~$a$. 
Therefore~$a$ and~$b$ 
are aligned 
if and only if 
every wall 
that belongs to 
an apartment 
that contains 
both~$a$ and~$b$ 
and contains 
either of these simplices 
also contains the other. 

%
The following proposition 
provides a way 
to measure 
the displacement 
between 
stabilizers of simplices 
with respect to~$\mathbf{d}$ 
defined in Lemma~\ref{lem:metricCOS}. 
It will be used 
to provide a formula 
for the scale of isometries.

\begin{prop}
\label{prop:delta-2-transitive=>distance-formula(cell_stab)}
Suppose that 
the action of a group $G$ 
on a building 
is Weyl-transitive and type-preserving. 
Let $a$ and~$b$ be 
two simplices 
with residues~$\mathcal{A}$ 
respectively~$\mathcal{B}$.  
Then 
\begin{equation}\label{eq:delta-2-transitive=>distance-formula(cell_stab)}
|G_b\colon G_b \cap G_{a}|= |G_b.\operatorname{proj}_\mathcal{B}(\mathcal{A})|\cdot q(b,a)\,.
\end{equation}
\end{prop}
\begin{proof}
Using 
part~(1) of~Lemma~\ref{lem:Stab-transitive(pairs(cells)_givenWdist)} 
and 
the orbit--stabilizer theorem, 
we see 
that 
$|G_b\colon G_b \cap G_{a}|$ 
is equal to 
the number of 
simplices 
of the same type 
as~$a$, 
whose residues~$\mathcal{R}$  
satisfy 
$\min\bigl(\delta(\mathcal{B},\mathcal{R})\bigr)=\min\bigl(\delta(\mathcal{B},\mathcal{A})\bigr)$.  
In order to count these, 
choose 
a fixed, reduced decomposition 
of~$\min\bigl(\delta(\mathcal{B},\mathcal{A})\bigr)$ 
and 
a set of 
chambers, 
$\mathcal{S}$ say, 
representing~$G_b.\operatorname{proj}_\mathcal{B}(\mathcal{A})$. 

The residue of 
each cell 
in~$G_b.a$ 
can be reached 
using 
a unique path 
leaving~$\mathcal{B}$ 
following 
a gallery 
of the chosen type 
that 
starts with 
a chamber 
in~$\mathcal{S}$.  
The number of 
these galleries 
equals~$|\mathcal{S}|\cdot q(\gamma)$,  
where~$\gamma$ 
is 
some gallery 
of the chosen type.  
Since 
the walls crossed 
by 
a gallery 
of the chosen type 
that ends in~$\mathcal{A}$ 
are precisely those 
that 
separate~$a$ from~$b$, 
the statement claimed 
is verified. 
\end{proof}

\begin{remark}\label{rem:compute-additional-factor_|G_b.a|}
~\\[-\baselineskip]
\begin{enumerate}
\item 
By Lemma~\ref{lem:Stab-transitive(pairs(cells)_givenWdist)} 
$|G_b.\operatorname{proj}_\mathcal{B}(\mathcal{A})|=|\mathcal{B}|/|\operatorname{proj}_\mathcal{B}(\mathcal{A})|$ 
if 
$(\Delta, \delta)$ 
is 
locally finite. 
\item 
The first factor 
on the right hand side 
of equation~(\ref{eq:delta-2-transitive=>distance-formula(cell_stab)})  
in the conclusion 
of Proposition~\ref{prop:delta-2-transitive=>distance-formula(cell_stab)} 
simplifies to~$1$ 
if and only if 
the pair~$(b,a)$ 
is aligned. 
\item 
An ordered pair 
beginning with 
a chamber 
is always aligned. 
\end{enumerate}
\end{remark}

For 
a simplex 
whose stabilizer 
is tidy 
for an element,  
the mutual position of 
the simplex 
and 
its image 
under that element 
is special. 
As a first step 
to understanding 
the relationship 
of these simplices 
we note 
the following.  

\begin{lemma}\label{lem:tidycellorbits=aligned}
Let~$G$ be 
a closed, Weyl transitive subgroup 
of the group of 
type preserving automorphisms of 
a locally finite building.  
If the stabilizer of a simplex~$a$ is tidy for $g\in G$, 
then the pairs $(g^n.a,g^m.a)$ 
are aligned 
for all~$n, m\in\mathbb{Z}$. 
\end{lemma}
\begin{proof}
Since $G_a$ is tidy 
for~$g$, 
it is also tidy for 
$g^l$ for each $l\in\mathbb{Z}$.
 Using Proposition~\ref{prop:delta-2-transitive=>distance-formula(cell_stab)} 
we have 
for 
fixed~$l\in\mathbb{Z}$ 
and every $n\in\mathbb{N}$ 
\begin{equation}
k_{a,nl}\cdot q(g^{nl}.a,a)
=|G_{g^{nl}.a}\colon G_{g^{nl}.a} \cap G_{a}|
= |G_{g^l.a}\colon G_{g^l.a} \cap G_{a}|^n
= \bigl(k_{a,l}\cdot q(g^l.a,a)\bigr)^n
\end{equation}
for some integers $k_{a,j}$ 
that are bounded above. 
Furthermore 
\begin{equation}
q(g^{nl}.a,a) \le \bigl(q(g^l.a,a)\bigr)^n\quad \text{for all }n\in\mathbb{N}
\end{equation}
which implies 
\begin{equation}
k_{a,nl}= k_{a,l}^n\,\frac{\bigl(q(g^l.a,a)\bigr)^n}{q(g^{nl}.a,a)}\ge k_{a,l}^n\quad \text{for all }n\in\mathbb{N}
\end{equation}
Thus $k_{a,l}=1$ for all $l\in\mathbb{Z}$, 
which means that 
the pairs~$(g^l.a,a)$ 
are aligned 
for each integer~$l$. 
Applying the automorphism~$g^{n-l}$ 
we see that the pairs~$(g^n.a,g^{n-l}.a)$ 
are aligned 
for each integer~$l$. 
Putting $m:=n-l$ 
we obtain 
our claim. 
\end{proof}

Elliptic elements 
may fix vertices 
whose stabiliser 
is then 
a tidy subgroup 
for the element. 
This can not happen 
for hyperbolic elements. 

\begin{corollary}\label{cor:Stab(vertex)=tidy=>isom=elliptic}
Let~$G$ be 
a closed, Weyl-transitive subgroup 
of the group of 
type-preserving automorphisms of 
a locally finite building.  
Then 
the stabiliser 
of a vertex 
is not tidy 
for any hyperbolic element 
of~$G$. 
\end{corollary}
\begin{proof}
By Lemma~\ref{lem:tidycellorbits=aligned} 
a simplex 
whose stabiliser 
is tidy for an element~$g$ 
must form an aligned pair 
with its image under~$g$. 
Since a pair of two vertices 
that are aligned 
are equal, 
any element 
that has 
the stabiliser of a vertex 
as a tidy subgroup 
must fix that vertex 
and hence is elliptic. 
The claim follows.  
\end{proof}

Our next result 
shows that 
we can always find 
a tidy subgroup 
for any element 
that equals the stabilizer 
of a simplex 
and compute 
the scale of the element 
from the walls separating 
the simplex 
and its image 
under the element 
in question. 

\begin{theorem}\label{thm:everything=straight}
Let~$G$ be 
a closed, Weyl transitive subgroup 
of the group of 
type-preserving automorphisms of 
a locally finite building.  
%
Let~$g$ be 
an element 
of~$G$, 
and $a$ be 
a simplex
which is 
\begin{itemize}
\item[(1)]
the carrier 
of some point in $\Min(g)$  
if~$g$ acts by elliptic isometries;  
%
\item[(2)]
the carrier 
of an open interval 
of an axis 
of~$g$ 
no points 
of which 
are separated by 
a wall 
if~$g$ 
acts by hyperbolic isometries. 
\end{itemize}
Then 
$(g^n.a,a)$ 
is aligned 
for all~$n\in\mathbb{Z}$, 
$G_{a}$ 
is tidy 
for~$g$ 
and 
$s(g)=q(a,g.a)$. 
\end{theorem}
\begin{proof} 
Since the claims 
are obvious 
if~$g$ 
acts by elliptic isometries, 
we may assume that 
$g$ acts by hyperbolic isometries 
in what follows. 

We begin 
by verifying 
the claim 
on alignment. 
Observe that 
lemmata~4.2 and~4.1 
in~\cite{geoFlats<CAT0-real(CoxGs+Tits-buildings)} 
imply that 
any wall 
in an apartment 
that contains 
a chosen axis 
of~$g$ 
with 
the wall  
containing 
an open interval 
of that axis 
contains said axis 
completely. 
By the defining property 
of~$a$,  
we conclude that 
$g^n.a$ and~$a$ 
are aligned 
for any~$n\in\mathbb{Z}$.

Applying Proposition~\ref{prop:delta-2-transitive=>distance-formula(cell_stab)} 
with~$b=g^n.a$, 
part~(2) of Remark~\ref{rem:compute-additional-factor_|G_b.a|} 
as well as 
Lemma~\ref{lem:separating_walls-compatibility}
we conclude 
that 
for every~$n\in\mathbb{N}$ 
\[
|g^nG_{a}g^{-n}\colon g^nG_{a}g^{-n}\cap G_{a}|=
|G_{g^n.a}\colon G_{g^n.a}\cap G_{a}|=
q(g^n.a,a)
\,.
\]

The positioning 
of~$a$ 
with respect to~$\Min(g)$ 
implies
\[
\forall n\in\mathbb{N}\colon q(a,g^n.a)=q(a,g.a)^n
\quad\text{which implies}
\]
\[
\forall n\in\mathbb{N}\colon 
|g^nG_{a}g^{-n}\colon g^nG_{a}g^{-n}\cap G_{a}|=
|gG_{a}g^{-1}\colon gG_{a}g^{-1}\cap G_{a}|^n\,, 
\]
showing that 
$G_{a}$ is tidy 
for~$g$ 
by Lemma~\ref{lem:tidyness-crit_powers}. 
Since 
this implies 
that 
$s(g)=q(a,g.a)$, 
the proof 
is  complete. 
\end{proof}

The following theorem 
provides 
an alternative description 
of a simplex 
with tidy stabilizer 
that does not mention 
an axis 
for the element 
in question.

\begin{theorem}\label{thm:straight-cell}
Let~$G$ be 
a closed group of Weyl-transitive, type-preserving automorphisms 
a thick, locally finite building. 
Let $g\in G$ 
and 
$a$ be a simplex  
that is aligned with~$g.a$ 
and 
such that 
$q(g.a,a)$ is minimal 
among all simplices 
with the same property. 
Then 
the stabilizer of~$a$ 
is tidy for~$g$ 
and $s(g)=q(g.a,a)$. 
\end{theorem}
\begin{proof}
This follows since 
tidy subgroups are minimizing 
and 
we already know 
that 
there is some simplex 
that is aligned with 
its image under~$g$ 
whose stabilizer is 
a tidy subgroup 
for~$g$. 
\end{proof}

Theorem~5.4 in~\cite{geo-char(flatG(auts))} 
follows immediately 
from Theorem~\ref{thm:straight-cell}. 
We have 
therefore found
a proof of that result 
that is 
both 
technically simple 
and  
valid 
in a much 
more general 
context.

\section{Interpretation as straight displacements}\label{sec:straight-displacements}
The aim of 
this section 
is 
to characterize 
a simplex 
whose stabilizer 
is tidy for 
an element 
in an algebraic manner 
in terms of 
the Weyl group 
of the building, 
see Theorem~\ref{thm:cotype-straight_displacement<=>simplex-Stab=tidy}. 
The apt term turns out to be the following. 

\begin{defn}
Let $I\subseteq S$ be a spherical subset. An element $w\in W$ is called \textit{$I$-straight} if $w$ is $(I,I)$-reduced,  
$wIw^{-1}=I$ and $\ell(w^n)=n\,\ell(w)$ for all $n\in\mathbb{N}$. 
\end{defn}

Since  
being straight 
is equivalent to 
being $\varnothing$-straight, 
$I$-straightness 
is a refinement 
of straightness. 
Furthermore, 
because $wIw^{-1}=I$ 
implies that 
conjugation by~$w$ 
fixes~$I$ pointwise, 
being $I$-straight 
implies 
being $J$-straight for each $J\subseteq I$.

We will require 
the following lemma.  
It is 
a special case of 
Theorem~2.1 in~\cite{dist-reg<buildings&struc-const<HeckeA}.

\begin{lemma}\label{lem:count}
Let~$G$ be 
a closed group of Weyl-transitive, type-preserving automorphisms 
a thick, locally finite building. 
Let $I\subseteq S$ be 
spherical 
and 
let $P_I$ denote 
the stabilizer of 
the cotype $I$ simplex 
of the base chamber.  
Then for 
all $(I,I)$-reduced 
elements $w\in W$ 
and $g\in G$
writing 
$Y(q)=\sum_{u\in Y}q_u$ for subsets $Y\subseteq W$ 
we have 
\begin{align*}
\left|\{hP_I\colon \delta(gP_I,hP_I)=w\}\right|=\frac{W_I(q)}{W_{I\cap wIw^{-1}}(q)}\,q_w,
\end{align*}
\end{lemma}

\begin{remark}
The paper~\cite{dist-reg<buildings&struc-const<HeckeA} 
assumes local finiteness 
and 
regularity 
of the building 
but 
no Weyl-transitive group action. 
Where all assumptions apply  
--- as in our context --- 
the following connection 
holds 
between 
the quotient $\frac{W_I(q)}{W_{I\cap wJw^{-1}}(q)}$
which appears 
in the general formula in Theorem~2.1 in~\cite{dist-reg<buildings&struc-const<HeckeA} 
and 
the corresponding factor 
in Proposition~\ref{prop:delta-2-transitive=>distance-formula(cell_stab)}. 
We may use  
part~(1) of 
Remark~\ref{rem:compute-additional-factor_|G_b.a|} 
and see that 
the numerator 
of that quotient 
is the cardinality of the $I$-residue 
and the denominator is 
the cardinality of 
the projection of the $J$-residue 
at Weyl-distance~$w$ 
on the $I$-residue. 
\end{remark}

The proposition below 
establishes 
the advertised characterisation 
of tidy subgroups 
for standard parabolics. 
The general case 
will be obtained 
by conjugating.

\begin{prop}\label{prop:I-straight_displacement<=>I-face-Stab=tidy}
Let~$G$ be 
a closed group of Weyl-transitive, type-preserving automorphisms 
a thick, locally finite building. 
Then 
the standard $I$-parabolic subgroup $P=P_I$ is tidy for $g\in G$ if and only if the element $w=\delta(P,gP)$ is $I$-straight. 
\end{prop}

\begin{proof}
$(\Longrightarrow)$ Suppose that $P$ is tidy for $g$. 
By Lemma~\ref{lem:tidyness-crit_powers} 
for all $n\in\mathbb{N}$ 
we have 
$|P:P\cap g^{-n}Pg^n|=|P:g^{-1}Pg|^n$, and it follows that $|P\backslash Pg^nP|=|P\backslash PgP|^n$ for all $n\in\mathbb{N}$. Write $w_n=\delta(P,g^n.P)$ (so that $w_n$ is necessarily $(I,I)$-reduced, and $w_1=w$).
We now argue 
similarly to 
the proof of Lemma~\ref{lem:tidycellorbits=aligned}.  
Lemma~\ref{lem:count} gives 
for all integers~$l$ 
\begin{equation}\label{eq:tidy-spheres}
\frac{W_I(q)}{W_{I\cap w_{nl}Iw_{nl}^{-1}}(q)}q_{w_{nl}}=\left(\frac{W_I(q)}{W_{I\cap w_lIw_l^{-1}}(q)}\right)^nq_{w_l}^n\,,
\end{equation}
from which we deduce 
for all integers~$l$ 
that 
\[
\frac{W_I(q)}{W_{I\cap w_{nl}Iw_{nl}^{-1}}(q)}
=\left(\frac{W_I(q)}{W_{I\cap w_lIw_l^{-1}}(q)}\right)^n\frac{q_{w_l}^n}{q_{w_{nl}}}
\ge \left(\frac{W_I(q)}{W_{I\cap w_lIw_l^{-1}}(q)}\right)^n\,.
\]
Since 
the left hand side 
is bounded above 
by~$W_I(q)$, 
we deduce that $W_{I\cap w_lIw_l^{-1}}(q)=W_I(q)$ for all $l$, and so $w_lIw_l^{-1}=I$ for all $l$. In particular, $wIw^{-1}=I$. 
Then formula~\eqref{eq:tidy-spheres} 
applied to~$l=1$ 
gives 
$q_{w_n}=q_w^n$, 
for all integers~$n$ 
and from this 
we deduce that $\ell(w_n)=n\,\ell(w)$ 
for all~$n$. 
All that remains is to prove that $\ell(w^n)=n\,\ell(w)$. We do this below by showing that in fact $w_n=w^n$ (and then use $\ell(w_n)=n\,\ell(w)$). 

Since $w_n=\delta(P,g^nP)$ there is a gallery of type $w_n$ joining a chamber of $P$ 
to a chamber of $g^nP$, and this gallery has minimal length amongst all such galleries. For each $i$, since $\delta(g^iP,g^{i+1}P)=\delta(P,gP)=w$ there is a reduced gallery of type~$w$ joining a chamber of $g^iP$ to a chamber of $g^{i+1}P$. Connecting these galleries together using connecting galleries whose types are in $W_I$ produces a gallery of type 
$$
w\cdot w_1\cdot w\cdot w_2\cdots w\cdot w_{n-1}\cdot w\quad\text{where}\quad w_1,\ldots,w_{n-1}\in W_I,
$$
from a chamber of $P$ to a chamber of $g^nP$. Of course this gallery might not be of reduced type. However the part of type $w\cdot w_1$ is of reduced type (because $w$ is $(I,I)$-reduced). Since $wIw^{-1}=I$ we may write $ww_1=w_1'w$ with $w_1'\in W_I$, and it follows that there is also a reduced gallery of type $w_1'\cdot w$ connecting the start and and chambers of the original gallery of type $w\cdot w_1$. So we make this distortion to produce a gallery of type 
$$
w_1'\cdot w\cdot w\cdot w_2\cdots w\cdot w_{n-1}\cdot w
$$
joining a chamber of $P$ to a chamber of $g^nP$. We can iterate this process to produce a gallery of type 
$$
w_1'\cdot w_2'\cdots w_{n-1}'\cdot\underbrace{w\cdot w\cdots w}_{n\text{ terms}}
$$
joining a chamber of $P$ to a chamber of $g^nP$. The initial leg of the gallery of type $w_1'\cdot w_2'\cdots w_{n-1}'$ stays completely inside the spherical parabolic $P_I$, and thus may be replaced by a gallery of type $w'\in W_I$ of length bounded by $\mathrm{diam}(W_I)$. Thus we have produced a gallery of type 
$$
w'\cdot\underbrace{w\cdot w\cdots w}_{n\text{ terms}}
$$
joining a chamber of $P$ to a chamber of $g^nP$. 

Suppose that for some $n_0\in\mathbb{N}$ we have $\ell(w^{n_0})< n_0\,\ell(w)$. It follows that the gallery of type $w\cdot w\cdots w$ ($n_0$ terms) above can be replaced by a gallery of length bounded by $n_0\,\ell(w)-1$. Now consider the case $n=kn_0$. We have a gallery of type $w'\cdot w\cdot w\cdots\cdot w$ (with $n$ factors $w$) from a chamber of $P$ to a chamber of $g^nP$. Break this gallery up as:
$$
w'\cdot\underbrace{\underbrace{(w\cdots w)}_{n_0\text{ terms}}\cdot\underbrace{(w\cdots w)}_{n_0\text{ terms}}\cdots\underbrace{(w\cdots w)}_{n_0\text{ terms}}}_{k\text{ terms}}.
$$
Each of the $\underbrace{(w\cdots w)}_{n_0\text{ terms}}$ parts of the gallery can be replaced by galleries of length at most $n_0\,\ell(w)-1$ connecting the same start and end galleries. Thus overall we produce a gallery of length at most
$$
\mathrm{diam}(W_I)+k\,(n_0\,\ell(w)-1)=n\,\ell(w)+\mathrm{diam}(W_I)-k=\ell(w_n)+\mathrm{diam}(W_I)-k.
$$
Thus choosing $k>\mathrm{diam}(W_I)$ we obtain a contradiction because we have a gallery of length strictly shorter than the minimum $\ell(w_n)$. Thus $\ell(w^n)=n\,\ell(w)$ for all $n\in\mathbb{N}$ and we are done.

$(\Longleftarrow)$ Suppose that $w=\delta(P,gP)$ is $(I,I)$-reduced, and that $wIw^{-1}=I$, and $\ell(w^n)=n\,\ell(w)$ for all $n\in\mathbb{N}$. Then
\begin{align*}
Pg^nP&\subseteq PgP\cdot PgP\cdots PgP=PwP\cdot PwP\cdots PwP=Pw^nP,
\end{align*}
and since double cosets are either disjoint or equal we have $Pg^nP=Pw^nP$. Therefore, using $wIw^{-1}=I$ and $\ell(w^n)=n\,\ell(w)$ we have
\begin{align*}
|P\backslash Pg^nP|&=|P\backslash Pw^nP|=\frac{W_I(q)}{W_{I\cap w^nIw^{-n}}(q)}q_{w^n}
=q_w^n=|P\backslash PwP|^n=|P\backslash PgP|^n.
\end{align*}
Thus $|P:P\cap g^{-n}Pg^n|=|P:g^{-1}Pg|^n$ for all $n\in\mathbb{N}$ and so $P$ is tidy for $g$ 
by Lemma~\ref{lem:tidyness-crit_powers}. 
\end{proof}

%
We now deduce  
the announced tidiness criterion 
for stabilizers of simplices 
in terms of Weyl-displacement 
in general. 

\begin{theorem}\label{thm:cotype-straight_displacement<=>simplex-Stab=tidy}
Let~$G$ be 
a closed group of Weyl-transitive, type-preserving automorphisms 
a thick, locally finite building 
and 
let~$g\in G$. 
Then 
the stabilizer of 
a simplex~$a$ is tidy for~$g$ 
if and only if 
$\delta(a,g.a)$ is $\operatorname{cotype}(a)$-straight. 
\end{theorem}
\begin{proof}
The simplex~$a$ 
is of the form~$hP$ 
with~$P$ 
a standard parabolic 
of $\operatorname{cotype}(a)$. 

By Propositon~\ref{prop:I-straight_displacement<=>I-face-Stab=tidy}, 
\[
\delta(a,g.a)=\delta(hP,g.hP)=\delta(P,h^{-1}gh.P)
\]
is $\operatorname{cotype}(a)$-straight 
if and only if 
$P$ is tidy for~$h^{-1}gh$. 
But this is equivalent to~$hPh^{-1}$, 
the stabilizer of~$hP=a$ being tidy for~$g$. 
The proof is complete. 
\end{proof}

\begin{definition}
For 
a type preserving automorphism~$g$  
of a building~$\Delta$ 
we 
call 
any straight element in $\{\delta(a,g.a)\colon a\text{ is a simplex of }\Delta\}$ 
a \emph{straight displacement} 
of~$g$.
\end{definition}

Theorem~\ref{thm:cotype-straight_displacement<=>simplex-Stab=tidy}
has implication for 
the values of the scale function. 

\begin{corollary}\label{cor:scale-values=q-values(straight-elements)}
Let~$G$ be 
a closed group of Weyl-transitive, type-preserving automorphisms 
a thick, locally finite building. 
Then
the scale of any element~$g$  
of~$G$ 
is 
the $q$-value 
of a straight displacement 
of $g$ 
(possibly trivial). 
\end{corollary}
\begin{proof}
Let~$g\in G$.  
Choose 
a simplex~$a$ 
in~$\Delta$ 
whose stabilizer is 
tidy for~$g$; 
this is possible  
by Theorem~\ref{thm:everything=straight}. 
By the same Theorem, 
we have $s(g)=q(a,g.a)$, which equals~$q_{\delta(a,g.a)}$. 
The element~$\delta(a,g.a)$ 
is $\operatorname{cotype}(a)$-straight 
by Theorem~\ref{thm:cotype-straight_displacement<=>simplex-Stab=tidy}, 
in particular 
it is straight. 
This proves our claim. 
\end{proof}

We will 
improved upon 
this corollary 
in Corollary~\ref{cor:scal-values(Weyl-trans-Aut)}.

\section{Geometric characterisation of simplices with tidy stabilizer}

We now establish, 
for a hyperbolic element,  
a connection 
between metric and combinatorial axes 
in the reverse direction to Theorem~\ref{thm:everything=straight}.
In a special case 
this result 
is part of the content of 
Lemma~5.4 in~\cite{orthForms(KMGs)=acyl-hyp}.

\begin{theorem}\label{thm:combinatorial-axis=metric-axis}
Let~$G$ be 
a closed group of Weyl-transitive, type-preserving automorphisms 
a thick, locally finite building. 
Let~$a$ be a simplex 
whose stabiliser 
is tidy 
for a hyperbolic element~$g\in G$. 
Then 
there is 
an axis 
of~$g$ 
that intersects 
the interior 
of the geometric realisation 
of~$a$.  
\end{theorem}
\begin{proof}
We may suppose 
that~$a$ is 
a face of 
the fundamental chamber, $c_0$. 
Then $w:=\delta(a,g.a)$ is $\operatorname{cotype}(a)$-straight. 
Hence, 
by Lemma~3.2 in~\cite{strongTitsAlt(<CoxeterG)}   
the element~$w$ 
has an axis, 
$L$ say,
in the Davis-realisation 
of the Coxeter complex 
of~$W$ 
that is contained in 
all walls 
of~$c_0$ 
of type~$i$ 
for all $i\in\operatorname{cotype}(a)$ 
and hence 
passes through
an interior point of~$a$.

It follows that 
the set $g^\mathbb{N}.c_0$ 
is $W$-isometric to 
a subset of an apartment, 
hence contained in an apartment. 
Therefore 
the geometric realization of 
this apartment 
contains 
the image of~$L$ 
which is 
an axis of~$g$ 
that intersects~$a$ 
and contains an interior point 
of it.  
\end{proof}

A direct, geometric construction 
of an axis 
without recourse to 
a straight displacement 
of the isometry 
is of interest. 

In terms of straight displacements 
we can reformulate 
the theorem 
as follows. 

\begin{corollary}
Let~$G$ be 
a closed group of Weyl-transitive, type-preserving automorphisms 
a thick, locally finite building. 
Then
every element~$g$ 
in~$G$ 
attains 
a straight Weyl-displacement~$\delta(a,g.a)$ 
only for simpices~$a$ 
whose relative interior intersects~$\Min(g)$. 
\end{corollary}
\begin{proof}
Theorem~\ref{thm:cotype-straight_displacement<=>simplex-Stab=tidy} 
implies that 
the stabilizer of 
the simplex~$a$ 
is tidy for~$g$. 
If~$g$ is elliptic, 
$\delta(a,g.a)=1$, 
the simplex~$a$ 
is fixed 
by~$g$  
and 
the claim holds. 
If~$g$ is hyperbolic, 
then  
there is an axis 
of~$g$ 
that intersects 
the relative interior of~$a$ 
by Theorem~\ref{thm:combinatorial-axis=metric-axis}.
\end{proof}

\begin{remark}\label{rem:uniqueness(straight-displacement)}
It is interesting 
to what extent 
straight displacements 
of a group element~$g$ 
are unique.  
Only the case of hyperbolic~$g$ 
is in question, 
since 
the unique Weyl-displacement of 
elliptic elements 
is the trivial element.
For hyperbolic~$g$ 
it can be shown that 
all such displacements 
on simplices 
that contain interior points 
of an axis 
of~$g$ 
are conjugate. 
As straight elements of Coxeter groups 
satisfy very tight restrictions, 
it is conceivable 
that 
stricter restrictions 
apply. 
\end{remark}


%
We can 
conclude 
that minimal displacement 
of points 
in the Davis-realisation 
is 
more or less 
the same as 
minimal displacement of 
their stabiliziers.

\begin{corollary}\label{cor:min-displacement-metric=min-displacement-scale}
Let~$G$ be 
a locally compact group 
that acts Weyl-transitively 
by type-preserving automorphisms 
of a thick building, 
and 
let~$g\in G$. 
Then 
for a point~$x$ 
in a Davis-realisation 
of the building 
the distance between $G_x$ and~$G_{g.x}$ 
is minimal 
if and only if 
$x$ does not lie 
on a wall 
that separates the attracting and repelling points of~$g$ at infinity 
and 
there is a point~$x'$ 
in the carrier of~$x$ 
such that 
the distance in the Davis-realisation 
between~$x'$ and~$g.x'$ 
is minimal. 
\end{corollary}
\begin{proof}
The condition on the walls 
eliminates precisely 
the points 
lying on 
essential walls 
for~$g$. 

For any point~$x$, 
the distance 
between~$G_x$ and~$G_{g.x}$ 
equals the distance 
between~$G_a$ and~$G_{g.a}$, 
where~$a$ is 
the carrier of~$x$. 
If this distance 
is minimal 
among all distances 
between~$G_y$ and~$G_{g.y}$ 
it must equal the minimal distance 
between a compact, open subgroup 
of~$G$ 
and its conjugate under~$g$ 
by Theorem~\ref{thm:everything=straight}
(that is, 
the distance 
is equal to 
$\log(s(g))+\log(s(g^{-1}))=2\,\log(s(g))$), 
and~$G_x$ is tidy for~$g$. 

Theorem~\ref{thm:combinatorial-axis=metric-axis} 
then shows 
that 
the carrier 
of~$x$ 
contains another point~$x'$ 
which lies on 
an axis 
for~$g$ 
as claimed. 

The reverse direction 
is similar, 
but easier, 
and 
is left to the reader. 
\end{proof}

This corollary 
has an interesting consequence, 
which will be elaborated on 
in more detail 
in Section~\ref{sec:apply}, 
compare Remark~\ref{rem:?q-type-restriction-from-straight}.

\begin{remark}\label{rem:min-displacement-metric=min-displacement-scale}
Notice that 
the distance between points 
according to 
the $\mathsf{CAT}(0)$-metric 
is roughly 
proportional to 
the number of walls 
the geodesic 
connecting these points 
crosses,  
while 
the distance between the stabilizers 
of those points 
is equal to 
the sum of 
the logarithms of 
the thicknesses of 
the walls 
the geodesic 
crosses. 
This corollary is 
therefore 
somewhat surprising, 
since, 
in the case 
when 
the building 
has walls  
with different thicknesses,  
it appears to 
say 
that 
(type-preserving) hyperbolic automorphisms 
will  pass through 
walls in a way 
which utilises 
the different available thicknesses 
in a balanced manner. 
\end{remark}

The formulation of 
Corollary~\ref{cor:min-displacement-metric=min-displacement-scale} 
would be much neater if (metric) minimal sets of group elements would always be subcomplexes. 
However, this is not the case 
as is illustrated by 
the following example. 

\begin{example}
Pictured below 
is  
a detail of 
the Davis-realisation 
(with standard choice of distances to walls) 
of 
an apartment with 
Coxeter group~$\mathrm{PGL}_2(\mathbb{Z})$ 
together with 
an illustration of 
the minimal set of 
a hyperbolic isometry 
(between the dotted gray lines 
with a fundamental domain 
under powers of the hyperbolic isometry 
filled in gray) 
together with 
the geometric realisation of 
a chamber 
and 
its image 
under the isometry 
(pictured in blue and green respectively). 

The walls 
are shown 
in black. 
You can 
also 
see 
the hexagons and squares 
used in the dual complex 
to represent 
the corresponding spherical residues. 
The Davis-realisation shown 
arises as 
the universal cover of 
the 4-6-12-tessalation 
of the plane 
with 12-gons removed 
subdivided 
by walls 
as indicated 
in the detail. 
The hyperbolic element 
is a lift of 
a translation symmetry 
of that tessalation.

You can find 
a picture of this Davis-complex 
overlaid over 
the Poincar\'e model of 
the hyperbolic plane 
as figure~12.11 on page~625 in~\cite{buildings_A&B}. 
That version 
shows 
an overview over 
the whole complex, 
but it is 
not metrically accurate. 
From that picture 
however, 
it is clear that 
the Davis-complex 
is 
a thickening of 
the Bass-Serre tree 
of~$\mathrm{PGL}_2(\mathbb{Z})$.  

\begin{center}
\includegraphics[width=14.6cm]{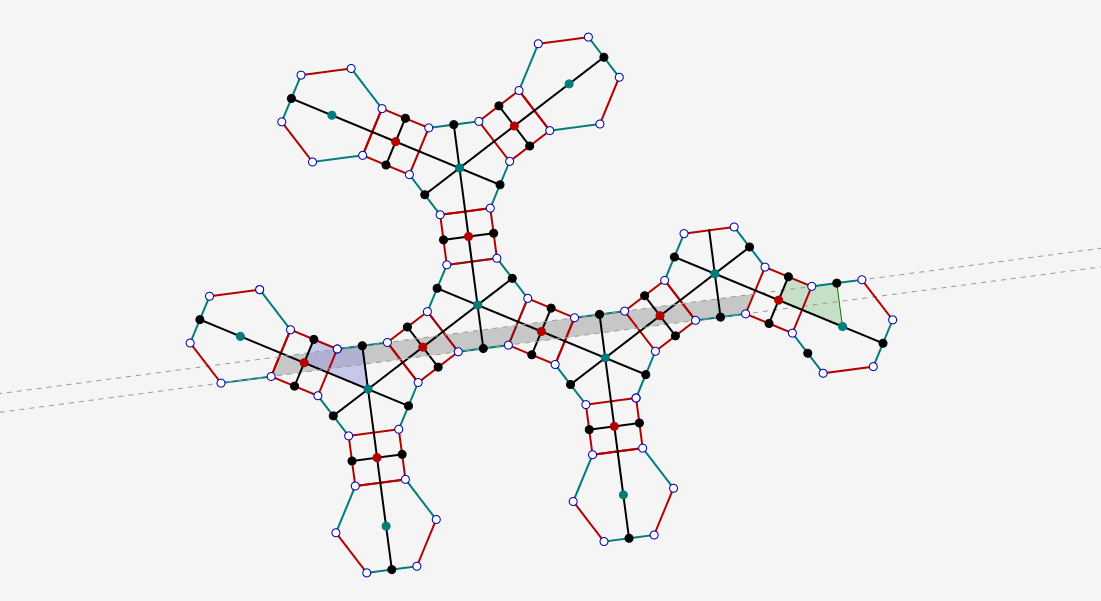}
\end{center}
It is seen 
here 
that 
the minimal set 
of the chosen translation 
is not a subcomplex. 
Indeed, 
an even more extreme 
trait 
can be seen  
in the same example 
by choosing 
the distances to the walls 
that factor in 
the construction of 
the Davis-realisation 
differently. 
If 
we increase 
the distance 
of the vertices 
of the quadrangles 
from the walls 
that do not intersect 
the hexagons 
sufficiently, 
the minimal set 
of the ``same'' isometry 
degenerates into 
a single geodesic. 

While 
this example 
lives within 
a thin building, 
the example  
can be thickened 
in a manner 
that is 
equivariant 
with respect to 
the hyperbolic isometry 
illustrated. 

 The following figure 
makes it easier to verify 
some of the claims 
made above. 
It shows a detail of 
the underlying 4-6-12-tessalation 
with walls extended into the 12-gons, 
which 
in the universal cover 
widen into 
a horoball-neighborhood of 
the cut-out vertices at infinity 
seen in 
the Poincar\'e-model. 
%
\begin{center}
\includegraphics[width=14.6cm]{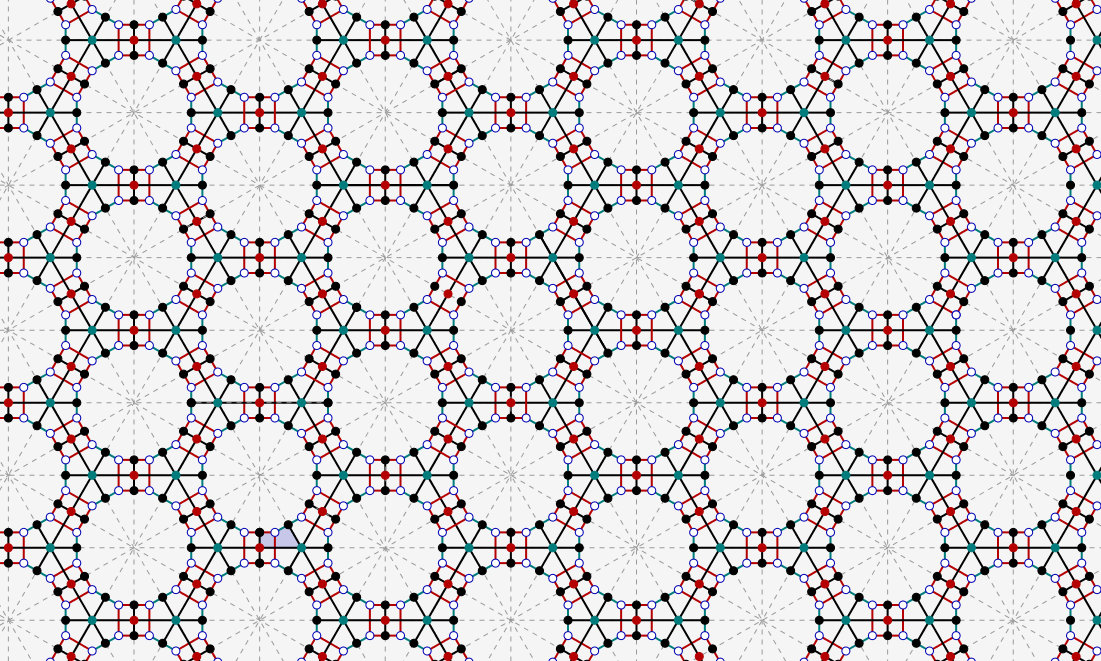}
\end{center}
\end{example}

\section{Applications}\label{sec:apply}
\hyphenation{Weyl-dis-place-ments Weyl-dis-place-ment}

\begin{prop}\label{prop:all-straight-displacements-exist}
Let~$G$ be 
a closed group of Weyl-transitive, type-preserving automorphisms 
a thick, locally finite building 
of type~$(W,S)$. 
Let~$I$ be 
a subset of~$S$ 
and~$w\neq 1$ 
an $I$-straight element 
of~$W$. 
Further, 
let~$a$ and~$b$ be 
 simplices 
of cotype~$I$ 
and 
Weyl-distance~$w$.  
Then 
there exists 
a hyperbolic element~$g$ 
of~$G$   
such that 
some axis 
of~$g$ contains 
interior points 
of~$a$ and~$b$ 
and 
$g.a=b$. 
\end{prop}
\begin{proof}
Choose a simplex~$x\neq b$ 
of cotype~$I$ 
at Weyl-distance~$w$ 
from~$a$. 
This is possible,  
since 
the number of elements 
of Weyl-distance~$w$ 
from~$a$  
is at least~$q_w\ge 2$. 

Choose chambers~$c_x$, $c_a$ and~$c_b$ 
such that 
$x$ is a face of~$c_x$, 
$a$ is a face of~$c_a$ 
and 
$b$ is a face of~$c_b$ 
and 
$\delta(c_x,c_a)=\delta(c_a,c_b)=w$. 
Then, 
by Weyl-transitivity 
of the action 
of~$G$,  
pick~$g\in G$ 
such that $g.c_x=c_a$ and~$g.c_a=c_b$. 

By construction, 
and because~$g$ preserves types, 
we have 
$g.a=b$ and $\delta(a,g.a)=\delta(a,b)=w$. 
By Theorem~\ref{thm:cotype-straight_displacement<=>simplex-Stab=tidy}
$G_a$ is tidy for~$g$. 
The simplices~$a$ and~$b$ 
are therefore aligned by Lemma~\ref{lem:tidycellorbits=aligned} 
and we may compute 
the scale of~$g$ 
using Proposition~\ref{prop:delta-2-transitive=>distance-formula(cell_stab)} 
to be~$q(a,b)=q_w>1$. 
We conclude that 
$g$ is hyperbolic. 
We obtain 
the claim 
about the position of 
an axis 
of~$g$ 
from Theorem~\ref{thm:combinatorial-axis=metric-axis}. 
The proof is complete. 
\end{proof}

\begin{remark}\label{rem:?q-type-restriction-from-straight}
We can now make Remark~\ref{rem:min-displacement-metric=min-displacement-scale} 
more precise. 
Since 
by the above Proposition, 
every non-trivial straight Weyl-displacement 
can be realised 
with the action of 
some hyperbolic element 
on an axis,  
Corollary~\ref{cor:min-displacement-metric=min-displacement-scale} 
appears to imply 
that 
there is a restricion 
on the thicknesses of different types of walls 
purely in terms of 
the Weyl group, 
limited only by 
the straight elements 
contained therein. 
\end{remark}

Note that 
straight elements 
are widespread 
in infinite Coxeter groups 
by Theorems~4.1.5 and~4.1.6 
in~\cite{asymp-behave(word-metrics>CoxeterGs)}.

Another interesting consequence 
of Proposition~\ref{prop:all-straight-displacements-exist} 
is 
the following corollary, 
which 
improves on 
Corollary~\ref{cor:scale-values=q-values(straight-elements)}. 

\begin{corollary}\label{cor:scal-values(Weyl-trans-Aut)}
Let~$G$ be 
a closed group of Weyl-transitive, type-preserving automorphisms 
a thick, locally finite building 
of type~$(W,S)$. 
Then 
the set of scale values of~$G$ 
equals $\{q_w\colon w\text{ is a straight element of } W\}$. 
In particular, 
$G$ is uniscalar 
if and only if 
$W$, 
and hence $\Delta$,  
are finite; 
or put differently, 
if and only if  
$\Delta$ is spherical. 
\end{corollary}
\begin{proof}
The scale of any element 
of~$G$ 
is 
the $q$-value 
of some straight element of 
the Weyl group 
by Corollary~\ref{cor:scale-values=q-values(straight-elements)}. 
Proposition~\ref{prop:all-straight-displacements-exist} 
then shows, 
that all the scale values listed 
in the statement 
above 
are 
in fact 
realised. 
This shows 
the first statement. 

To derive 
the last statement, 
note that 
$G$ is uniscalar 
if and only if 
all of its elements 
are elliptic. 
The latter 
necessarily happens 
if~$\Delta$ is finite. 
If 
on the other hand, 
$\Delta$ is infinite, 
then its Weyl group 
is infinite as well. 
Now 
every infinite Coxeter group 
contains non-trivial straight elements; 
any Coxeter element in 
an infinite, irreducible component 
will do, 
see~\cite[Corollary~F]{conj-classes&straight-el<CosGs}. 
The remaining claim follows 
form this  
and the part of our claim 
that we already proved. 
%
\end{proof}

The important class 
of rank one isometries 
can also be described 
by their straight displacements 
as shown 
in the next result. 

\begin{prop}
Let~$G$ be 
a closed group of Weyl-transitive, type-preserving automorphisms 
a thick, locally finite building~$(\Delta,\delta)$ 
of type $(W,S)$.  
Then 
a hyperbolic element~$g\in G$ 
is not a rank one isometry 
of the Davis-realisation 
of~$(\Delta,\delta)$
if and only if 
some, equivalently every, straight displacement  
of~$g$ 
is contained 
in $W_I \times W_J$, where either~$W_I$ and~$W_J$ are both infinite, or~$W_I$ is affine and~$W_J$ is finite.
\end{prop}
\begin{proof}
Let~$w$ be 
a straight displacement 
of~$g$ 
and~$L$ 
an axis 
of~$g$ 
that 
passes through 
a simplex~$a$ 
whose Weyl-displacment $\delta(a,g.a)$ 
equals~$w$. 
Further, 
let~$A$ be an apartment 
that contains~$L$; 
it will then also contain~$a$. 
Choose a chamber~$c$ 
of~$A$ 
that contains~$a$ 
and 
let~$\rho$ be 
the retraction 
onto~$A$ 
centered at~$c$. 

Then $\rho\circ g$ restricts to 
an automorphism 
of~$A$ 
which is given by~$w$ 
if~$A$ 
is identified with
the Coxeter complex 
of~$W$ 
by mapping~$c\in A$ to~$1\in W$.

Assume that 
$w$ is contained in 
$W_I \times W_J$, where either~$W_I$ and~$W_J$ are both infinite, or~$W_I$ is affine and~$W_J$ is finite. 
Then 
we conclude that~$w$ 
stabilizes $W_I \times W_J$, 
and thus 
--- varying the apartment~$A$ --- 
the isometry~$g$ 
stabilizes 
a residue 
of type $W_I \times W_J$ 
in~$\Delta$. 
By Theorem~5.1 in~\cite{rk1-isom(buildings)&q-morph(KMGs)} 
$g$ is not a rank one isometry. 

If, 
conversely,~$g$ 
is not a rank one isometry, 
the arguments of 
the implication $(\romannumeral3)\Rightarrow (\romannumeral1)$ 
in the proof of  
Theorem~5.1 in~\cite{rk1-isom(buildings)&q-morph(KMGs)} 
show that 
every straight displacement 
of~$g$ 
is contained in 
$W_I \times W_J$, where either~$W_I$ and~$W_J$ are both infinite, or~$W_I$ is affine and~$W_J$ is finite. 
The proof 
is complete. 
\end{proof}

Furthermore, 
we suspect 
that 
Weyl-displacements 
of arbitrary simplices 
under a hyperbolic isometry~$g$ 
can be computed from 
the position of the simplex 
relative to 
an axis 
of~$g$ 
and 
a suitable straight displacement 
of~$g$ 
on this axis.

\end{document}